\documentclass[a4paper,11pt,oneside]{amsart}
\usepackage{geometry}

\usepackage[english]{babel}
\usepackage{amsmath}
\usepackage[all]{xy}
\usepackage{amsthm}
\usepackage{amssymb}
\usepackage{mathtools}
\usepackage{xcolor}
\usepackage{enumitem}
\usepackage{xpatch}
\usepackage{multicol}  
\usepackage{setspace}
\setdisplayskipstretch{}

\makeatletter
 \def\@textbottom{\vskip \z@ \@plus 14pt}
 \let\@texttop\relax
\makeatother

\definecolor{linkblue}{RGB}{1,1,190}
\definecolor{citered}{RGB}{190,1,1}

\usepackage[linkcolor=linkblue
           ,urlcolor=linkblue
           ,citecolor=citered
           ,ocgcolorlinks        
           ,bookmarksopen=True,  
           ]{hyperref}
\usepackage[ocgcolorlinks]{ocgx2}
\makeatletter
  \AtBeginDocument{
    \hypersetup{
      pdftitle  = {\@title},
    }
  }
\makeatother
  
\usepackage[capitalize]{cleveref}

\usepackage{microtype}

\newtheorem{teor}{Theorem}[section]
\newtheorem{defi}[teor]{Definition}
\newtheorem{defis}[teor]{Definitions}
\newtheorem{prop-defi}[teor]{Proposition-Definition}
\newtheorem{lemma}[teor]{Lemma}
\newtheorem{prop}[teor]{Proposition}
\newtheorem{cor}[teor]{Corollary}

\theoremstyle{remark}
\newtheorem{remark}[teor]{Remark}

\newtheorem{example}[teor]{Example}

\DeclareMathOperator{\add}{add}

\DeclareMathOperator{\Hom}{Hom}
\DeclareMathOperator{\im}{im}
\DeclareMathOperator{\End}{End}

\DeclarePairedDelimiter{\card}{\lvert}{\rvert}

\setlist[enumerate,1]{label=\textup{(\arabic*)}, ref=\textup{(}\arabic*\textup{)},
  itemsep=0.5em plus 0.15em minus 0.05em,
  topsep=0.5em plus 0.15em minus 0.05em,
  leftmargin=0.75cm}
\setlist[enumerate,2]{label=\textup{(\roman*)}, ref=\textup{(}\roman*\textup{)}
  itemsep=0.5em plus 0.15em minus 0.05em,
  topsep=0.5em plus 0.15em minus 0.05em}
\setlist[itemize, 1]{itemsep=0.5em plus 0.15em minus 0.05em,
  topsep=0.5em plus 0.15em minus 0.05em, leftmargin=0.75cm}

\newlist{equivenumerate}{enumerate}{1}
\setlist[equivenumerate,1]{%
  label=\textup{(\alph*)},
  ref=\textup{(}\alph*\textup{)},
  itemsep=0.5em plus 0.15em minus 0.05em,
  topsep=0.5em plus 0.15em minus 0.05em,
  leftmargin=0.75cm
}

\newlist{proofenumerate}{enumerate}{1}
\setlist[proofenumerate,1]{%
  itemsep=0.5em plus 0.15em minus 0.05em,
  topsep=0.5em plus 0.15em minus 0.05em,
  wide, labelindent=0pt
}

\xpatchcmd{\paragraph}{\normalfont}{{\normalfont\bfseries}}{}{}

\newcommand{\defit}[1]{\textsf{#1}}

\newcommand{\V} {\rm {V}}
\newcommand{\CFM}  { \rm CFM } 
\newcommand{\RCFM}  { \rm RCFM  }

\DeclareMathOperator{\Tr}{Tr}

\DeclareMathOperator{\Vmon}{V}
\newcommand{\st}{\operatorname{st}}
\newcommand{\LL}{\operatorname{\mathcal L}}

\title[realization of monoids with countable sum]{realization of monoids with countable sum}

\author{Zahra Nazemian}
\address{University of Graz\\
         Department of Mathematics and Scientific Computing\\
         NAWI Graz\\
          Heinrichstra\ss e 36\\
          8010 Graz, Austria}
\email{zahra.nazemian@uni-graz.at}



\begin{document}

\begin{abstract}
  
For every infinite cardinal number $\kappa$, $\kappa$-monoids and their realization have recently been introduced and studied in \cite{me}. 
As stated in \cite[Corollary 4.7]{me}, a $\kappa$-monoid $H$ has a realization to a ring $R$ 
if there exists an element $x \in H$ such that $H$ is $\aleph_1 ^{-}$-braided over $\text{add}(\aleph_0 x)$, and $\text{add}(\aleph_0 x)$,  as $\aleph_0$-monoid, 
 has a realization to $R$. Furthermore, $H$ has a realization to hereditary rings if there exists an element $x \in H$ such that $H$ is 
 braided over $\text{add}(x)$. These prompt an investigation into when $\aleph_0$-monoids have realizations. 
In this paper, we discuss the realization of $\aleph_0$-monoids and provide a complete characterization for the realization of two-generated ones in hereditary Von Neumann regular rings.
\end{abstract}

\date{}

\maketitle

\section{Introduction}

If $R$ is a  ring, then $V (R)$, the class isomorphism of finitely genereted projective right modules,  
 is a reduced commutative monoid with order-unit. In their seminal papers Bergman \cite {Bergman} and 
Bergman and Dicks \cite {BergWar} show that: for every reduced commutative monoid with order-unit and every
 field $k$, there exists a hereditary $k$-algebra $R$ with  $V(R) \cong H.$
 In the past two decades, the investigation of projective modules that are not necessarily finitely generated has been a focal point among researchers, as evidenced by works such as those by Herbera, P\v{r}ihoda, and Wiegand, R. Antonie, et al. \cite{HerberaPrihodaWiegand23, PavelHerbera, ARAOTHERS}. 
For every cardinal number $\kappa$, $\kappa$-monoids, a class of commutative monoids where the sum of $\kappa$ many elements is allowed, have been recently introduced and studied 
in   \cite{me}.
When $R$ is a ring, isomorphism classes of projective (right) modules generated by at most $\kappa$ many elements, 
denoted by $V ^{\kappa} (R)$,  
is an example of a $\kappa$-monoid. 
A $\kappa$-monoid $H$ is said to have \defit{realization} if it is isomorphic as a $\kappa$-monoid to $V^{\kappa}(R)$, for some ring $R$, see \cite[Section 2]{me}. 
Although $\kappa$-monoids are reduced monoids,  they do not generally have realizations.

Besides the study of properties of $\kappa$-monoids, the realization problem for them has been intensively studied in \cite{me}. 
As it is recalled in the following, it turned out that the focus should be on the study of the realization problem for $\aleph_0$-monoids. 
Let us recall some definitions before delving deep into the problem, and prior to that, we adopt the conventions from \cite{me}. Namely,
 we use infinite ordinals and cardinals, but to avoid set-theoretical difficulties, we always consider such cardinals up to a fixed bound $\kappa$.
As underlying axiom system we have in mind the usual ZFC (with informal classes).
As usual in this context, we use the von Neumann definition of ordinals, and a cardinal is the smallest ordinal of a given cardinality (but cardinals are by default just considered as sets, forgetting the well-order).
We make liberal use of the axiom of choice, to ensure that the basic properties of cardinal arithmetic are well-behaved.
See \cite[Chapter 5]{roitman90} for background.

If $x \coloneqq (x_i)_{i \in \kappa} \in H^{\kappa}$ is a family indexed by elements of $\kappa$ over some set $H$, and $\Sigma \colon H^{\kappa} \to H$ is a map, we use the notation
\[
  \sum_{i \in \kappa} x_i \coloneqq \Sigma\big( (x_i)_{i \in \kappa}\big),
\]
mimicking the typical summation notation. 

\begin{defi} \label{def:kappa-monoid}
  Let $\kappa$ be an infinite cardinal.
  A $\kappa$-monoid is a set $H$ together with an element $0 \in H$ and a map $\Sigma \colon H^{\kappa} \to H$ such that the following conditions are satisfied.
  \begin{enumerate}[label=\textup{(A\arabic*)},leftmargin=1.25cm]
  \item \label{a:trivial} If $x=(x_i)_{i \in \kappa} \in H^\kappa$ with $x_i=0$ for all $i \ne 0$, then $\sum\limits_{i \in \kappa} x_i=x_0$.
  \item \label{a:flatten}
    If $(x_{i,j})_{i,j \in \kappa} \in H^{\kappa \times \kappa}$ and $\pi \colon \kappa \times \kappa \to \kappa$ is a bijection, then
    \[
      {\sum_{i \in \kappa}} {\sum_{j \in \kappa}} x_{i,j} = {\sum_{k \in \kappa}} x_{\pi^{-1}(k)}.
    \]
  \end{enumerate}
\end{defi}

Naturally, as it is explained in details in \cite[Section 2]{me},  when  $H$ is a  $\kappa$-monoid, we can define a binary operation on $H$  which makes $H$ a commutative  (reduced) monoid. 
Indeed,  sum of two elements  $a, b\in H$ is defined to be 
 $ \sum\limits_{i \in \kappa} x_i$, where $x_0 = a, x_1 = b$ and all other $x_i = 0$. 
 Similarly,
when $I$ is a subset of $ \kappa$ and $(y_i)_{i \in I}$ is a family of elements in $H$, we can define 
$ \sum\limits _{i \in I} y_i :=  \sum \limits_{i  \in \kappa} x_i  $, where $x_i = y_i$ for every $i \in I$ and $y _i = 0$ other places. Using this, 
   one can see that  when 
  $ \alpha \leq \kappa$ are infinite  cardinal  numbers, then a  $\kappa$-monoid is an $\alpha$-monoid. In particular, every 
  $\kappa$-monoid is $\aleph_0$-monoid, where $\aleph_0:= \{0, 1, \cdots \}$. 
  Moreover, 
   by  $\alpha a$, where $a \in H$ and $\alpha \leq \kappa$ is any cardinal number, we mean sum of $\alpha$-many copies of $x$, that is $ \sum \limits_{i \in \alpha} a  $. 
Given a monoid $H$ and  element $x \in H$, we write
\[
  \add(x) \coloneqq \{\, y \in H : \text{there exist } z \in H,\, n  \geq 1 \text{ with } y+z=nx \,\}
\]
for the set of all summands of multiples of $x$. 
An element  $u \in H$ is called an \defit{order-unit} if  $H = \add(u)$. 
It is not difficult to see that $\kappa$-monoids are  reduced (that is, any equation of the form $a+b=0$ implies $a=b=0$), see \cite[Lemma 2.8]{me}.

Let $\mathcal{P}_{\text{fin}}$ be a class of finitely generated projective  (right) modules over some (unital, associative) ring $R$. 
Then the set of isomorphism classes together with the operation induced by the direct sum form a monoid, denoted by $\V(R)$. 
Every  element in 
 $\V(R)$ is in form $[P]$, the class isomorphism of a  finitely generated projective module $P$. 
Obviously,  $\V(R)$  is a commutative reduced monoid, where
$[R]$ is an  order-unit. A ring $R$ is called right (resp., left) \defit{hereditary} if every right (resp., left) ideal is a projective $R$-module. 
A ring which is both right and left hereditary is called  a hereditary ring.
By Bergman and Dicks' result for every reduced commutative monoid $H$  with order-unit and every
 field $k$, there exists a hereditary $k$-algebra $R$ with  $V(R) \cong H.$
 Similarly, 
 $\Vmon^{\kappa}(R)$ is the $\kappa$-monoid of isomorphism classes of projective (right) modules generated by at most $\kappa$-many elements. 
 Although $\kappa$-monoids are reduced, they may not have a  realization even if they have an order-unit (see \cite[Section 2]{me}). 
 For a $\kappa$-monoid $H$ and $x \in H$, $\operatorname{add}(\aleph_0 x)$ is an $\aleph_0$-monoid, which plays an important role in the realization problem:

\begin{teor} {\rm (\cite [Corollary 4.7]{me})}\label{amix}
 
 A $\kappa$-monoid $H$ has realization  (that it is isomorphic to $\Vmon^{\kappa}(R)$ for a ring $R$)  iff there exits $x \in H$ such that:
 \begin{enumerate}
 \item  $\add(\aleph_0 x)$ has a realization as $\aleph_0$-monoid. 
 \item \begin{enumerate}

 \item   For every $h \in H$, there exist a family of elements $ (x_i) _{i \in \kappa}$ in $\add(\aleph_0 x)$  such that $h =  \sum\limits_{i  \in \kappa} x_i $.
 \item When ever  $\sum\limits_{i  \in \kappa} x_i  =  \sum\limits_{i  \in \kappa} y_i $, there exist 
  indexed partitions $(I_{\mu})_{\mu \in \kappa}$ and $(J_\mu)_{\mu \in \kappa}$ of $\kappa$ with $\card{I_\mu}$,~$\card{J_\mu} \leq  \aleph_0 $, and families $(u_\mu)_{\mu \in \kappa}$ and $(v_\mu)_{\mu \in \kappa}$ in $\add(\aleph_0 x)$ such that $v_\mu=0$ whenever $\mu$ is a limit element, and 
    \[
      \sum\limits_{i \in I_\mu} x_i = v_\mu + u_\mu,
      \qquad\text{and}\qquad
      \sum_{j \in J_\mu} y_j = v_{\mu+1} + u_{\mu} \qquad \text{for all } \mu \in \kappa.
    \]
 
 \end{enumerate}
 \end{enumerate}
 
\end{teor}

Therefore, the realization problem for $\kappa$-monoids can be restricted to $\aleph_0$-monoids. In this manuscript, we discuss the realization of $\aleph_0$-monoids, or, 
in other words, monoids with countable sums. The structure of the paper is as follows: Section 1 delves into the realization problem concerning monoids featuring countable sums.
We show that if an $\aleph_0$-monoid has a realization, the $x$ given in Theorem \ref{amix} can be better understood, and we refer to such $x$ as a braiding element.
We observe that for an $\aleph_0$-monoid having a realization to a  hereditary ring is equivalent to having a realization to a ring whose projective 
right modules are direct sums of finitely generated ones (we say such a ring is PDFG).\\
In Section 3, we focus on two-generated  $\aleph_0$-monoids. 
Recall that a set $X$ of elements of an  $\aleph_0$ monoid $H$ is called a {\defit {generating set}} if the smallest $\aleph_0$-submonoid 
of $H$ containing $X$ is equal to $H$. In this 
case we write 
$H =   \langle X   \rangle _{\aleph _0}$. For example,  Condition (2)(i) in Theorem 1.2  asserts that $H = \langle \add(\aleph_0 x) \rangle_{\aleph_0}$.
 We call an $\aleph_0$-monoid $H$ \defit{cyclic} if there exists an element $x \in H$ such that $H = \langle x \rangle_{\aleph_0}$, 
meaning every element in $H$ is of the form $\alpha x$, where $0 \leq \alpha \leq \aleph_0$. As we see in Lemma \ref{again}, a cyclic  $\aleph_0$-monoid $H$  has 
realization iff 
$\aleph_0 x \neq nx$ for every $n < \aleph_0$ iff
 $H$ has realization to a hereditary ring. 
A  $\kappa$-monoid $H$ is called \defit{two-generated} if it has a generating set with two elements, say $H =  \langle \{x_1, x_2\}   \rangle _{\aleph _0} $. 
That means 
every $y \in H$ can be represented as $y = \sum\limits_{k \in \aleph_0} y_k$ with $y_k$ equal to $x_1$,  $x_2$ or $0$. Since the order of the elements in the
 family $(y_k)_{k \in \aleph_0}$ does not matter, it is convenient to represent such a family as $\alpha X_1 + \beta X_2$, where $0 \leq \alpha \leq \aleph_0$ is the number of elements in the family that are equal to $x_1$, and $0 \leq \beta \leq \aleph_0$ is the number of elements that are equal to $x_2$. We call such a representation $\alpha X_1 + \beta X_2$ a \defit{form}
  for $y$ (based on generating set $\{x_1, x_2\}$). Clearly the equality $y = \alpha x_1 + \beta x_2$ also holds. 
  We say that $\alpha X_1 + \beta X_2$ is an \defit{infinite form} if at least one of $\alpha$, $\beta$ is infinite, and a \defit{finite form} otherwise. Depending on $H$, an element may have both infinite and finite forms. In \cite[Theorem 5.3]{me}, a characterization is provided for when a two-generated $\aleph_0$-monoid has realization to hereditary rings. Namely,

\begin{teor} {\rm \cite  [Theorem 5.3] {me}}\label{hereditarycasecor}
  A non-cyclic $\aleph_0$-monoid $H$ generated by two elements $x_1$,~$x_2$ is isomorphic to $V^{\aleph_0}(R)$ for a
   hereditary ring $R$, if and only if the following conditions hold for $1 \leq i \neq j \leq 2$.
  \begin{enumerate}[label=\textup(\roman*\textup),leftmargin=2.5em] 
  \item \label{hcc:threeinf} \label{hcc:first} $n x_i + \aleph_0 x_j = \aleph_0 x_i + \aleph_0 x_j$ with $n$ finite implies $\aleph_0 x_j = \aleph_0 x_i +\aleph_0 x_j$ and $x_i \in \add(x_j)$.
  \item \label{hcc:twoinf} If $x_i \not \in \add(x_j)$ and $m x_i + \aleph_0 x_j = n x_i + \aleph_0 x_j$ with $m$,~$n$ finite, then there exist finite $k$,~$k'$ such that $m x_i + k x_j = n x_i + k' x_j$.
  \item \label{hcc:finite-infinite} \label{hcc:last} An element of $H$ cannot have both finite and infinite forms.
  \end{enumerate}
\end{teor}

We note that a two-generated $\aleph_0$-monoid $H$ cannot simultaneously be realized as both a PDFG and a not-PDFG ring. We examine the properties of those with realizations as not-PDFG rings; refer to Propositions \ref{notpdfg} and \ref{1st-case-lemma} for detailed discussions.
To deepen our comprehension of a two-generated $\aleph_0$-monoid $H = \langle {x_1, x_2} \rangle_{\aleph_0}$, we introduce the concept of "type" for these monoids, which hinges on whether $x_i$ belongs to  $\add(x_j)$ or not. Consequently, three distinct possibilities emerge, each of which we scrutinize. 
We establish that if $H$ possesses a realization, it can only conform to one of these types. However, the converse remains uncertain; see Theorem \ref{typeseriuse} for further insights. \\
A ring $R$ is called von Neumann regular (\defit{VNR}) if, for every element $a\in R$, there exists an element $x \in R$ such that $axa=a$. 
When $R$ is VNR, $\V(R)$ becomes particularly interesting as it forms a refinement monoid. Recall that a monoid $H$ is considered a \defit{refinement} if $a + b = c + d$ implies 
that $a = x + y$, $b = z + t$, $c = x + z$, and $d = y + t$, for some 
$x, y, z, t \in H$. This situation can be represented in the form of a square:

\begin{center}
\medskip
\begin{tabular}{|c||c|c|c|}\hline
& c & d \\\hline\hline
a & x & y \\\hline
b & z & t \\\hline
\end{tabular}
\end{center}

For years, as discussed in \cite{pere} for a comprehensive survey, the question of whether a refinement-reduced monoid has a realization to a VNR remained open. Finally, in \cite{AraBosaPardo}, P. Ara, J. Bosa, and E. Pardo provided an affirmative answer to this question in the case
 where the monoid is finitely generated.
In Section 4, we provide a complete understanding of when a two-generated $\aleph_0$-monoid $H = \langle {x_1, x_2} \rangle_{\aleph_0}$ has a realization to VNR hereditary rings. 
It is not difficult to see that this is equivalent to this that the submonoid of $H$ generated by $x_1, x_2$, denoted by $\langle {x_1, x_2} \rangle$,  being a refinement monoid and $H$ being braiding over $\text{add}(x_1 + x_2)$. We provide 
another equivalent characterization by introducing additional conditions to Theorem \ref{hereditarycasecor}; for further details,  see Theorem \ref{last}.




\smallskip\paragraph{Acknowledgments.}
The author was supported by the Austrian Science Fund (FWF): P 36742-N.  She  would like to thank  Alfred Geroldinger for the mentoring during running this project and also Roozbeh Hazrat and Gene Abrams 
for introducing some papers.

\section{ Realization of $\aleph_0$-monoids: Some background} 
Before delving into the realization problem, we gain a deeper understanding of 
 $\V^{\aleph_0}(R)$, where $R$ is a ring, viewed as a monoid. Since $\V^{\aleph_0}(R)$ is a reduced monoid with 
order-unit $F:= R^{(\aleph_0)}$, it is isomorphic to $\V(S)$ for a hereditary ring $S$, as shown by Bergman and Dicks' realization. However, as we will see in the next lemma, there is a monoid isomorphism between $\V^{\aleph_0}(R)$ and $\V (\End(F_R))$.

Thinking of elements of $F$ as columns, we can identify $\End (F_R)$ with the column finite matrix  ring $\CFM_{\aleph_0}(R)$, consisting of $\aleph_0 \times \aleph_0$ matrices, where each column has only finitely many nonzero entries.

\begin{lemma}\label{columnfinite}
For a ring $R$, let $S := \CFM_{\kappa}(R)$. Then, $\V^{\aleph_0} (R) \cong \V(S)$ as an isomorphism of monoids.
 In particular, when $R$ is a VNR, $S \cong \RCFM(R)$, the ring of row and column finite matrix rings over $R$.
\end{lemma} 

\begin{proof}
For a given right $R$-module $M$, let $\add(M_R)$ be the class of modules, such as $N$, where $N$ is isomorphic to a direct summand of a finite direct sum of $M$. There exists an equivalence between the full subcategory $\add(M)$ and finitely generated projective $\text{End}(M)$-module, given by $\text{Hom}_R(M, -)$ or, conversely, by $- \otimes_{\text{End}(M)} M$ (see \cite[Theorem 2.35]{Facchini}). Now, for $F := R^{(\aleph_0)}$, it is evident that $\add(F)$ is precisely the class of countably generated projective modules.
 Therefore,
 the map $\phi : \V^{\aleph_0}(R) \to \V(S)$ defined by $\phi([P]) = [\text{Hom}_R(F, P)]$ is an isomorphism of monoids. The rest of the lemma follows from \cite{AraPardoPerera}.
\end{proof}

  If we rephrase the braiding concept from \cite{me} for $\aleph_0$-monoids, we get:

\begin{defi} {\rm  \cite [Definition 3.1]{me}} \label{def:braiding}
  Let $H$ be an $\aleph_0$-monoid and  $X$ be a submonoid (not necessarily an  $\aleph_0$-submonoid) of $H$.  
  \begin{enumerate}
  \item Two families $(x_i)_{i \geq 0}$ and $(y_i)_{i \geq 0}$ in $X$ are {\defit{braided}} if there exist indexed partitions $(I_{n})_{n \geq 0}$ and $(J_n )_{n \geq 0}$ of
   $\aleph_0$ with $\card{I_n}$,~$\card{J_n}$ finite, and families $(u_n )_{n \geq 0}$ and $(v_n)_{n \geq 0}$ in $X$ such that $v_0 = 0$, and 
    \[
      \sum_{i \in I_n} x_i = v_n + u_n,
      \qquad\text{and}\qquad
      \sum_{j \in J_n} y_j = v_{n+1} + u_{n} \qquad \text{for all } n \geq 0.
    \]
  \item  $H$ is {\defit{braided over $X$}} if $H = \langle X \rangle_{\aleph_0}$, and all families $(x_i)_{n \geq 0}$, $(y_j)_{n \geq 0}$ in $X$ 
  with $ \sum\limits_{i \in \aleph_0} x_i  =  \sum\limits_{j \in \aleph_0}  y_j$ are  braided.
  \end{enumerate}
\end{defi}

In this article, while working with $\aleph_0$-monoids, our objective is to offer a clearer understanding of \cite[Theorem 4.3]{me} and introduce a slightly different concept of braiding,
 from what is defined in Definition \ref{def:braiding}. The proof follows a similar structure to that of \cite[Theorem 4.3]{me}, but we specify in Definition \ref{def:braiding} that 
 $I_{n}$ and $J_n$ are intervals in $\aleph_0$, particularly in the form $\{ k \mid m_1 \leq k \leq m_2 \}$, where $m_1, m_2 \in \aleph_0$.

\begin{teor} \label{directsumofinfinite}
Suppose that the $\aleph_0$-monoid $H$ has a realization. Then, there exists an element $x$ such that whenever $\{x_i \mid i \geq 0\}$ and $\{y_i \mid i \geq 0\}$ are families in $\add(x)$, then
$$
\sum_{{i \in \aleph_0}} x_i = \sum_{{i \in \aleph_0}} y_i
$$
if and only if there exist integers $0 = m_0 < m_1 < \cdots$ and $0 \leq n_0 < n_1 < \cdots$ and sequences $u_0, u_1, u_2, \cdots$ and 
$v_0 = x_0, v_1, v_2, \cdots$ in $\add(x)$ such that
\begin{align*}
&\sum_{i = 0}^{n_0} y_i = v_0 + u_0  \\
&\sum_{i = m_j + 1}^{m_{j + 1}} x_i = u_ j  + v_{j + 1}, {\text{ for every } } j \geq 0, {\text {and}}\\
&\sum_{i = n_j + 1}^{n_{j + 1}} y_i = u_{j + 1}  + v_ {j + 1}, {\text{ for every } } j \geq 0.
\end{align*}
\end{teor}

\begin{proof}
$ \Longleftarrow $ is clear. 
\\
For the other side, 
let $R$ be the realization of $H$, and 
$\phi : H \to \V^{\aleph_0}(R)$
 be an $\aleph_0$-isomorphism and 

\begin{equation}
\sum_{{i \in \aleph_0}} x_i = \sum_{{i \in \aleph_0}} y_i \tag{Eq 2.1} \label{Eq 2.1}
\end{equation}

Take $x$ to be the element such that $\phi(x) = [R]$. 
Then clearly whenever $t \in \add(x)$ and  $\phi(t) = [X]$, then $X$  is finitely generated. 
 For each $i \geq 0,$ let $A_i$ and $B_i$ be modules such that   $[A_i] = \phi(x_i)$ and $[B_i ]= \phi(y_i)$. Following \ref{Eq 2.1} and using the fact that 
$\phi$ is an $\aleph_0$-isomorphism, we  conclude that

\begin{equation}
\bigoplus_{i \geq 0} A_i \cong  \bigoplus_{i \geq 0} B_i \tag{Eq 2.2} \label{Eq 2.2}
\end{equation}

Let  $f:  \bigoplus_{i \geq 0}  A_i \to  \bigoplus_{i \geq 0} B_i$ be an isomorphism and show its inverse with 
 $g:  \bigoplus_{i \geq 0}  B_i \to   \bigoplus_{i \geq 0} A_i$. 
We have $m_0 = 0$. Since $f(A_0)$ is
 finitely generated, there exists $n_0$ such that $f(A_0) \leq \bigoplus_{i = 0}^{n_0} B_i$. Since $A_0$ is a direct summand of $ \bigoplus_{i \geq 0} A_i$, so is its
  image in $ \bigoplus_{i = 0}^{n_0} B_i$, see \cite [Lemma 2.1, p.33]{fac}. Thus, $\bigoplus_{i = 0}^{n_0} B_i = f(A_0) \bigoplus T^{y}_0$, for some $R$-module $T^{y}_0$, and so

\begin{equation}
\bigoplus_{i = 0}^{n_0} g(B_i) = A_0 \bigoplus g(T^{y}_0).  \tag{Eq 2.3} \label{Eq 2.3}
\end{equation} 

Then put $u_0 := \phi^{-1}([T^y_0])$. 
Note that  $g$ is  isomorphim and so $g(B_i) \cong B_i$ and also 
$\phi$ is a monoid isomorphism and so 
we get to the first equation asked in the theorem, that is $\sum_{i = 0}^{n_0} y_i = v_0 + u_0$, where $v_0 := x_0$. \\

By a similar argument, since $\bigoplus_{i = 0}^{n_0} g(B_i) $ is finitely generated submodule of $ \bigoplus_{i \geq 0} A_i$, 
 there exist
 $m_1 > m_0 = 0$ and $R$-module $T^{x}_1$ such that 
$\bigoplus_{i = 0}^{m_1} A_i = \bigoplus_{i = 0}^{n_0} g(B_i) \oplus T_1 ^{x}$. 
 Therefore,
 $\bigoplus_{i = 0}^{m_1} A_i = A_0 \bigoplus g(T^y_0) \bigoplus T^x_1$,
 and so $\bigoplus_{i = {m_0 + 1}}^{m_1} A_i \cong g(T^{y}_0) \bigoplus T^{x}_1$,
  noting that $m_0 = 0$. Since $g$ is an isomorphism,  $g(T^y_0) \cong T^y_0$ and consequently 
 $u_0 = \phi^{-1}([g(T^{y}_0)])$. Now put $v_1 := \phi^{-1}([T^{x}_1])$.  Thus:

$$\sum_{i = m_{0} + 1}^{m_{1}} x_i = u_ 0  + v_{1}$$

Now, by induction, suppose that for $k \geq 1$, integers $n_0, \cdots, n_{k-1}$ and $m_0, \cdots, m_k$, and elements $u_0, \cdots, u_{k-1}, v_0, \cdots, v_k \in \add(x)$, and a 
finitely generated module $T^x_k$ are found, such that for them all equations given in the theorem hold, and moreover:
\begin{equation}
\bigoplus_{i=0}^{m_k} A_i = g(\bigoplus_{i=0}^{n_{k-1}} B_i) \bigoplus T^x_k \tag{Eq 2.4} \label{Eq 2.4}
\end{equation}
 and $v_k = \phi^{-1}([T^x_k])$. Now, we find integers $n_k$, $m_{k+1}$, and $u_k$, $v_{k+1} \in \text{add}(x)$ satisfying the conditions of the theorem.
 Since $f(\bigoplus_{i = 0}^{m_{k}}  A_i)$ is finitely generated, there exists $n_k > n_{k - 1}$ and finitely generated module $T^y_k$ such that 
\begin{equation}  
\bigoplus_{i = 0}^{n_{k}}  B_i = f(\bigoplus_{i = 0}^{m_{k}}  A_i) \bigoplus T^y_k \tag{Eq 2.5}  \label{Eq 2.5}
\end{equation}

Then $ u_k:= \phi^{- 1} ( [T^y_k] ) $ and from \ref{Eq 2.4}, we conclude that 
$$
\bigoplus_{i = 0}^{n_{k}}  B_i = \bigoplus_{i = 0}^{n_{k - 1}} B_i \bigoplus f(T^x_k)  \bigoplus T^y_k 
$$

Consequently, $
\sum_{i = n_{k-1} + 1}^{n_{k}} y_i = u_k + v_k  $. 

Now, similarly, since $g (\bigoplus_{i = 0}^{n_{k}}  B_i ) $ is finitely generated, 
there exists $ m_{k + 1} > m_{k}$ and finitely generated module $ T^x_{k + 1}$  such that

 $$ \bigoplus_{i = 0}^{m_{k + 1} }  A_i = g(\bigoplus_{i = 0}^{n_{k}} B_i) \bigoplus T^x_{k + 1}$$

That is we set $ v_{k + 1}:= \phi^{- 1} ( [T^x_{k+ 1}] )  $ 
and using \ref{Eq 2.5}
we conclude  that
$
\sum_{i = m_{k} + 1}^{m_{k + 1}} x_i = u_k + v_{k + 1}  $. 

\end{proof}

\begin{defi}\label{again1}
The element $x$ given in Theorem \ref{directsumofinfinite} is called a \defit{braiding element}. The proof of Theorem \ref{directsumofinfinite}
is showing that 
$[R]$ is a braiding element in $\V ^{\aleph_0} (R)$. 
\end{defi}

As a corollary, we have:

\begin{cor} \label{summeryit}
  If $x$ is a braiding element and $\aleph_0 z_1 = \aleph_0 z_2$ for some $z_1, z_2 \in \add(x)$, then $\add(z_1) = \add(z_2)$. 
\end{cor}

\begin{proof}
  Let $\aleph_0 z_1 = \aleph_0 z_2$. Then, by Theorem \ref{directsumofinfinite}, considering all $x_i$ in
   Theorem \ref{directsumofinfinite} as $z_1$ and all $y_i$  as $z_2$, we find $n_0 \geq  0$ and $u_0$ such that
  \[
    (n_0 + 1) z_2 = \sum_{i = 0}^{n_0} z_2 = z_1 + u_0.
  \]
  That is, $z_1 \in \add(z_2)$.  By a similar argument, we have  $z_2 \in \add(z_1)$. Consequently, $\add(z_1) = \add(z_2)$. 
\end{proof}

\begin{defis}
We need to define and recall the following concepts: 
\begin{enumerate}
\item If $R$ is a ring whose projectives are direct sums of finitely generated ones, we say $R$ is  {\defit {PDFG.}} 
\item 
We call   an $\aleph_0$-monoid $H$  \defit{hereditary} if it has a realization to a  hereditary ring $R$.
\item Recalling 
 from  \cite{me}, if  
$H$ is an  $\aleph_0$-monoid and $x \in H$, then $H$ is called braided over $\add(x)$, whenever:
\begin{enumerate}

 \item    $\add( x)$ is a generating set.
 \item Whenever  $\sum\limits_{i \in \aleph_0} x_i  =  \sum\limits_{i \in \aleph_0} y_i $, there exist 
  partitions $(I_{\mu})_{\mu \in \aleph_0}$ and $(J_\mu)_{\mu \in \aleph_0}$ of $\aleph_0$, 
  and families $(u_\mu)_{\mu \in  \aleph_0}$ and $(v_\mu)_{\mu \in  \aleph_0}$ in $\add(x)$ such that $v_0=0$, each $I_{\mu}$ and $J_\mu$ is finite such that the followings hold: 
    \[
      \sum_{i \in I_\mu} x_i = v_\mu + u_\mu, {}
      \sum_{j \in J_\mu} y_j = v_{\mu+1} + u_{\mu} \qquad \text{for all } \mu \in \aleph_0.
    \]
\end{enumerate}
\end{enumerate}
\end{defis}

Clearly, when $x$ is a braiding element, see Definition \ref{again1},   and $H = \langle  \add(x)  \rangle_{\aleph_0}$, then $H$ is braided 
over $ \add(x) $ according to Part (3) of the above definition. The converse is also true:

\begin{cor} \label{useful}
For an $\aleph_0$-monoid $H$, the followings are equivalent:
\begin{enumerate} [itemsep=5pt]
    \item $H$ has a realization to a PDFG ring.
    \item There is a braiding element $x$ such that $\add(x)$ is a generating set for   $H$.
    \item  There is an element $x \in H$ such that $H$ is braided  over $\add(x)$. 
    \item $H$ is hereditary. 
\end{enumerate}
\end{cor}
\begin{proof}
 \begin{proofenumerate}
    \item [$(1) \implies (2)$.] 
    Let $H$ have a realization to a PDFG ring. Then we have an $\aleph_0$-monoid isomorphism, denoted by $f$, from $H$ to $\V(R)$ for a PDFG ring $R$. 
    Let $x := f^{-1} ([R])$. Then, since 
    $[R]$ is a braiding element in $\V ^{\aleph_0} (R)$, so is      
     $x$  in $H$. Moreover, since $f$ is an $\aleph_0$-morphism, so is $f^{-1}$, and therefore:
    \[
    \begin{aligned}
        H & = f^{-1}(\V ^{\aleph_0} (R)) \\
        & = f^{-1}(  \langle \add ([R])  \rangle _{\aleph_0}) \\
        & =   \langle  \add (f^{-1}([R]))  \rangle_{\aleph_0} \\
        & =  \langle  \add(x)  \rangle_{\aleph_0}.
    \end{aligned}
    \]
    Where, the second equality holds because $R$ is  PDFG. This proves (2).
    
    \item[$(2) \implies (3)$] is clear. 
    
     \item[$(3) \implies (4)$]  is   \cite [Corollary 4.7]{me}. 
     
    \item[$(4) \implies (1)$.] It is known that over a hereditary ring, every projective is a direct sum of finitely generated ones,  see for example \cite{Albrecht}.  
\end{proofenumerate}
\end{proof}

\begin{remark}
Note that $x$ is an order-unit in $\operatorname{add}(x)$, and according to  Bergman's realization, the monoid isomorphism
from the reduced commutative monoid $M$ with order unit $x$ to 
$\V(R)$
 can take $x$  to 
$[R]$. Consequently, when $H$ is braided over $\operatorname{add}(x)$, then by 
\cite[Corollary 4.7]{me}, we find  an isomorphism of $\aleph_0$-monoids $H$ and $\mathcal{V}^{\aleph_0}(R)$ such that $x$ maps to $[R]$. This means $x$ is a braiding element.
\end{remark}



\section{Type of $\aleph_0$-monoids}

We begin this section with the realization of cyclic $\aleph_0$-monoids.
When $M$ is a module and  $ 0 \leq \alpha \leq \aleph_0 $, then $ \alpha M$ denotes 
direct sum of $ \alpha$ many copies of $M$.

\begin{lemma}\label{again}
For a cyclic $\aleph_0$-monoid $H := \langle x \rangle_{\aleph_0}$, the following are equivalent:
\begin{enumerate} 
\item $H$ has a realization. 
\item $H$ is a hereditary. 
\item $\aleph_0 x \neq nx$ for every $n < \aleph_0$.
\end{enumerate}
\end{lemma}

\begin{proof}\begin{proofenumerate}
 \item [$(1) \implies (3).$] 
Let $H := \langle x \rangle _{\aleph_0}$ be a cyclic $\aleph_0$-monoid, $R$ be a realization of $H$, and $f: H \to V^{\aleph_0}(R)$ be an isomorphism of $\aleph_0$-monoids.
There exists a right $R$-module $N$ such that $f(x) = [N]$. Consequently, every element of $V^{\aleph_0}(R)$ is equal to $[\alpha N]$, for some $\alpha \leq \aleph_0$, and therefore $N$ is a finitely generated module. If $\aleph_0 x = nx$ for some finite integer $n$, then $\aleph_0 N \cong nN$, which is not possible as $N$ is finitely generated. 

\item [$(3) \implies (2).$]  Suppose that (3) holds. The submonoid of $H$ generated by $x$, denoted by $\langle x \rangle$, is a reduced, commutative monoid with $x$ as an order-unit. 
Thus, by Bergman's realization, there is a hereditary ring $R$ such that $\langle x \rangle$ has a realization to
 $V(R)$. Let $g: \langle x \rangle \to V(R)$ be an isomorphism of monoids. Define $\bar{g}: H \to V^{\aleph_0}(R)$ such that $\bar{g}(\aleph_0 x): = \aleph_0 g(x)$ and $\bar{g}(n x): =  g(nx)$, for $n < \aleph_0$. Using the hypothesis (3), $\bar{g}$ is well-defined. Moreover, since $R$ is hereditary, all countably generated projective modules are direct sums of finitely generated ones. Thus, $\bar{g}: H \to V^{\aleph_0}(R)$ is an isomorphism.

\item [$(2) \implies (1)$]  is trivial.  
 \end{proofenumerate} 
\end{proof}

\textit{Due to the above lemma, from now on in this section, all $\aleph_0$-monoids (including $V^{\aleph_0}(R)$) are assumed to be non-cyclic and two-generated.}

Let us start with the following fact: 

\begin{lemma} \label{notpdfg}
 $H$ cannot be realized simultaneously to  both a PDFG and a non-PDFG ring.
\end{lemma}

\begin{proof}
We need to show that if $R$ is PDFG and $S$ is not PDFG, then $\V^{\aleph_0}(R)$ and $\V^{\aleph_0}(S)$ are not isomorphic. Let $f: \V^{\aleph_0}(R) \to \V^{\aleph_0}(S)$ be an $\aleph_0$-isomorphism, and let $P_1$ and $P_2$ be finitely generated projective modules such that $\V^{\aleph_0}(R) = \langle \{ [P_1], [P_2]\} \rangle_{\aleph_0}$. Let $P'_1$ and $P'_2$ be $S$-modules such that
 $f([P_1]) = [P'_1]$ and $f([P_2]) = [P'_2]$, respectively. Then $\V^{\aleph_0}(S) = \langle \{ [P'_1], [P'_2] \} \rangle_{\aleph_0}$ and $P'_1, P'_2$ are both finitely generated, as shown by \cite[Lemma 5.1]{me}.
 This implies that $S$ has to be a PDFG ring.
\end{proof}

If $P$ is a projective module, the \defit{trace} of $P$ is
\[
  \Tr(P) \coloneqq \sum_{f \in \Hom(P,R)} \im(f).
\]
Recall that $\Tr(P)$ is an idempotent ideal of $R$, and that $\Tr(P)$ is the least among all ideals $I$ such that $P=PI$, see for example \cite{PavelandPuninski} and \cite{whitehead}.

\begin{prop} \label{1st-case-lemma}
Suppose that $H$ is generated by $x_1$ and $x_2$ and has realization to a non-PDFG ring. 
Then  there exist $1 \leq  i \neq  j \leq 2$  such that:
\begin{enumerate} [itemsep=5pt]
    \item   $x_j \notin {\add}(x_i)$.
    \item $\aleph_0 x_i + x_j =\aleph_0 x_i$.

\item $x:= x_1 + x_2$ is not a braiding element. 

    \item  $x_i \in \add(x_j)$ if  and only if $\aleph_0 x_j = \aleph_0  x_i$.

    \item If $x_i \notin {\add}(x_j)$ and  $\alpha X_i + \beta X_j $ and $ \alpha' X_i + \beta' X_j$ are two forms of an element, then  $\alpha, \alpha'$ are both finite or infinite.
\end{enumerate}
\end{prop}

\begin{proof}
Suppose that
$H$ has realization to $V^{\aleph_0} (R)$, where $R$ is a  non-PDFG ring and $f: H \to V^{\aleph_0} (R)$ is an $\aleph_0$-monoid isomorphism. 
Then $ \{ f(x_1), f(x_2) \}$ is a generating set 
for  $V^{\aleph_0} (R)$. 
So   $[R] =n  f(x_1)  +m f(x_2) $, for some  finite integers $n, m  \geq 0$. 
Since $R$ is not a PDFG, there exists a projective module $P$ such that $P$ is not direct sum of finitely generated ones.  On the other hand, $[P]$  must be
sum of  some (finite or infinite) copies of $ f(x_1) $ and  $f(x_2)$. 
Suppose that  $P_1$ and $P_2$ are projective $R$-modules that 
 $[P_1] =  f(x_1)$ and $[P_2] = f(x_2)$.
This implies that $P_1$ or $P_2$ can not be finitely generated. Suppose that $P_2$ is not finitely generated. We set $j: = 2$ and $i:= 1$. Then: 

\begin{enumerate}[itemsep=5pt]

 \item  We have   $[R] =n  f(x_i)   $
and consequently $f(x_i)$ is finitely generated and $f(x_j)$ is not in add$(f(x_i))$. Consequently $x_j \notin \text{add}(x_i)$.

    \item  $f(x_j)$ is class isomorphism of a pojective module  and  so $f(x_j) + y = \aleph_0 [R]$, for some $y \in \V^{\aleph_0} (R)$. Now from \cite[Lemma 2.8]{me}, 
we conclude that $f(x_j) + \aleph_0 [R] = \aleph_0 [R]$.  Replacing $[R]$  with 
$n f(x_i)$, we conclude that  $f(x_j) + \aleph_0 f(x_i) = \aleph_0 f(x_i)$
and so $x_j +   \aleph_0  x_i= \aleph_0  x_i $.

\item Since $H = \langle \{x_1, x_2\} \rangle_{\aleph_0}$, and $x_1, x_2 \in \mathrm{add}(x)$,  where $x:= x_1 + x_2$, it is clear that $H = \langle \mathrm{add}(x) \rangle_{\aleph_0}$. Then, using
  Lemmas \ref{useful} and \ref{notpdfg}, we conclude that $x$ is not a braiding element.

\item  Let $x_i \in \text{add}(x_j)$. Then $m x_j = x_i + y$, for some $y \in H$ and finite integer $m \geq 1$. 
Consequently,  $m f(x_j) = f(x_i) + f(y)$. Now,  following equivalence condition of (c) and (d) in \cite [Proposition 5.4] {me}, we conclude that 
$ \aleph_0 ( m f(x_j))$, which is equal to  $  \aleph_0  f(x_j)$,   represents the  free module $\aleph_0 R$. Therefore,  bearing in mind that $[R] = n f( x_i)$, we have  $\aleph_0  f(x_j) =  \aleph_0   [R] = \aleph_0 f(x_i )$.
Consequently $ \aleph_0  x_j =   \aleph_0  x_i$.  \\

Conversely, let $ \aleph_0  x_j =   \aleph_0  x_i$. Then we have $ \aleph_0  f( x_j) =   \aleph_0  f(x_i) $. 
 Therefore, we have an isomorphism  
$g: \oplus _{i \geq 1} P_j \to \oplus _{i \geq 1} P_i$. Since $P_i$ is finitely generated $R$-module, there are $y_1, \cdots y_k \in P_i$ such that $P_i = y_1 R + \cdots + y_k R$. For each  $y_i$, there exists $m_i \in \oplus _{i \geq 1} P_j  $ 
such that $g(m_i) = y_i$. Now each $m_i$ has finite support and we conclude that there exists an epimorphism from a  direct sum of finitely many copies of $P_j$, say $m$ copies, 
 to $P_i$.  Since 
 $P_i$ is projective, this epimorphism 
splits and we have that $m P_j \cong  P_i \oplus Y$, for some 
projective module $Y$.  That means  $m f(x_j)  =  f(x_i)  +  [Y]$  and so  $x_i \in \add(x_j)$. 
    
    \item  
Suppose that $\alpha x_i+ \beta x_j = \alpha' x_i + \beta' x_j$, thus
 $\alpha P_i \oplus  \beta P_j \cong  \alpha' P_i   \oplus     \beta' P_j$. 
 Using Part (4), $\aleph_0 P_i \ncong \aleph_0 P_j$, and since $\aleph_0 P_i$ is free, we conclude that $\aleph_0 P_j$ is not free. Therefore, the trace of $P_j$ is not equal to $R$, see \cite[Proposition 5.4]{me}.
 Consequently, by tensoring $\alpha P_i \oplus  \beta P_j \cong  \alpha' P_i   \oplus     \beta' P_j$
 with $R/ I$ where $I = \Tr(P_j)$, we have that $\alpha (P_i / P_i I ) \cong \alpha' (P_i / P_i I )$. Since $P_i$
and consequently $P_i / P_i I $  is finitely generated module, we conclude that $\alpha$ and $\alpha'$ are both finite or both infinite.
\end{enumerate}
\end{proof}

\begin{defi}
    Let $H$ be a noncyclic $\aleph_0$-monoid generated by two elements. Then $H$ is classified into the following types:

    \begin{enumerate} [itemsep=5pt]
        \item {\rm Type (1)}: There exists a generating set $\{x_1, x_2\}$ such that $ x_1 \notin \add(x_2)$ and  $ x_2 \notin \add(x_1)$.

        \item {\rm Type (2)}: There exists a generating set $\{x_1, x_2\}$ such that $x_i \in \add(x_j)$ for 
some $i \neq j$, but $x_j  \notin \add(x_i)$.

        \item {\rm Type (3)}: There exists a generating set $\{x_1, x_2\}$ such that  $x_i \in \add(x_j)$ for every $1 \leq i, j \leq 2$. Equivalently 
        $\add(x_1) = \add(x_2) $. 
    \end{enumerate}
\end{defi}

\begin{teor} \label{typeseriuse}
Suppose that  $H$   is of {\rm Type (i)}, for $1 \leq i \leq 3$. 
If $H$ has realization and $1 \leq j\neq i \leq 3$,  then $H$  is not of {\rm Type (j)}. 
\end{teor}

\begin{proof}
It is enough to prove the following: 

\begin{enumerate}[itemsep=5pt]
\item [{\rm (i)}] If $H$ is of Type (3), then it cannot be of Type (1) or (2). 
 
 \item [{\rm (ii)}] If  $H$ is of Type  (2), then it can not be of Type (1). 
\end{enumerate} 

\begin{proofenumerate} [itemsep=5pt]
\item [{\rm Proof of (i):}] 
If $H$ is of Type (3), then there exists a generating set $\{x_1, x_2\}$ such that $\add(x_1) = \add(x_2)$. 
Then there is only one infinite form regarding $\{x_1, x_2\}$, see Part (2) of \cite[Corollary 5.5]{me}.  
Now let $\{y_1, y_2\}$ be another generating set. 
For each $1 \leq i \leq 2$, there exist $0 \leq \alpha_i, \beta_i \leq \aleph_0$ such that 
 $y_i = \alpha_i x_1 + \beta_i x_2$. 
Since $y_i$ is nonzero, for each $i$, either $\alpha_i$ or $\beta_i$ is nonzero. 
Therefore, $\aleph_0 y_i = \aleph_0 x_1$; in particular, this implies $\aleph_0 y_1 = \aleph_0 y_2$.\\
Now let $R$ be a realization of $H$. Since $\add(x_1) = \add(x_2)$, by Part (1) of Proposition \ref{1st-case-lemma}, $R$ is  a PDFG ring. So, 
by Corollary  \ref{useful}, there is a braiding element, say $x \in H$. 
There exist $0 \leq m_1, m_2 \leq \aleph_0$ such that
 $x = m_1 x_1 + m_2 x_2$. 
 Then since $x$ is nonzero, one of $m_1$ or $m_2$ is nonzero. 
So either $x_1$ or $x_2$ is in $\add(x)$. Since $\add(x_1) = \add(x_2)$, we have that $\add(x_1) = \add(x_2) \subseteq \add(x)$. \\
For each $1 \leq i \leq 2$, there exist $0 \leq \gamma_i, \delta_i \leq \aleph_0$ such that 
 $x_i = \gamma_i y_1 + \delta_i y_2$. 
If both $\gamma_1, \gamma_2$ are zero, then $H = < x_1, x_2>_{\aleph_0} \subseteq < y_2 >_{\aleph_0}$ and consequently, $H$ would be cyclic, which is a contradiction. 
Therefore, one of $\gamma_1$ or $\gamma_2$ is nonzero. That means $y_1 \in \add(x_1)$ or $y_1 \in \add(x_2)$. Thus, $y_1 \in \add(x)$. 
Similarly, at least one of $\beta_1$ or $\beta_2$ is nonzero, and so $y_2 \in \add(x)$.  So both $y_1, y_2 \in \add(x)$. \\
Finally, we are able to use Corollary \ref{summeryit} and from $\aleph_0 y_1 = \aleph_0 y_2$ and $y_1, y_2 \in \add(x)$, we conclude that $\add(y_1) = \add(y_2)$. 
\item [{\rm Proof of (ii):}] 
Suppose that $H$ is of Type (2) with the generating set $\{x_1, x_2\}$ such that $x_1 \in \add(x_2)$ and $x_2 \notin \add(x_1)$.
Let $H$ also be of Type (1). Then, there exists a generating set $\{y_1, y_2\}$ with $y_1 \notin \add(y_2)$ and $y_2 \notin \add(y_1)$.
Let $R$ be realization of $H$ and  $f: H \to \V^{\aleph_0} (R)$ be an $\aleph_0$-monoid isomorphism. There exist projective $R$-modules $M_1, M_2, N_1, N_2$
such that  $ [M_i] = f(x_i) $ and $  [N_i] = f(y_i) $, for every $1 \leq i \leq 2$.
   Then  $\{[M_1], [M_2]\}$ and  $\{[N_1], [N_2]\}$ are  generating sets for  $\V^{\aleph_0} (R)$. 
Moreover, since $x_1 \in \add(x_2)$, there exist finite $n_1 \neq 0$ and also  $1 \leq n_2 \leq \aleph_0$  and 
$ 0 \leq n_3 \leq \aleph_0$  such that $n_1 x_2 = n_2 x_1 + n_3 x_2$.
Now, let's consider two cases: whether $H$ has a realization to a PDFG or not.

Firstly, suppose that $H$ has a realization to a non-PDFG.
As $R$ is not a PDFG, either $M_1$ or $M_2$ is not finitely generated. Thus, $[R] = n [M_i]$ for some $1 \leq i \leq 2$ and finite $n \geq 1$, implying that $M_i$ must be finitely generated.
From 
 $n_1 x_2 = n_2 x_1 + n_3 x_2$, we have  $n_1 f(x_2) = n_2 f(x_1) + n_3 f(x_2)$. As $n_2$ is nonzero, if $M_1$ is not finitely generated, then $M_2$ is not finitely generated either. This is a contradiction because for $1 \leq i \leq 2$, $[R] = n [M_i]$ implies that $M_i$ must be finitely generated. This shows that $M_1$ is finitely generated, and $M_2$ is not. Therefore, $[R] = n [M_1]$. Consequently, for $x := f^{- 1} ([R])$,  we have $\add(x_1) = \add(x)$.

Since $R$ is not a PDFG ring,  one of $N_i$ is not finitely generated. Thus there exists $1 \leq i \leq 2 $ such that 
 $x = n' y_i$ for some finite $n'$. Without loss of generality, we suppose that $ x = n' y_1$. Thus, $\add(y_1) = \add(x) = \add(x_1)$.

There exist $0 \leq \alpha, \beta \leq \aleph_0$ such  that 
$y_2 = \alpha x_2 + \beta x_1$.  
 We claim that $\beta = 0$. Otherwise, $x_1 \in \add(y_2)$, and so $\add(x_1) \subseteq \add(y_2)$, which is a contradiction since $y_1 \in \add(x_1)$. 
 Consequently,
 $y_2 = \alpha x_2$.
  That is, $\add(x_2) \subseteq \add(y_2)$. Now, $y_1 \in \add(x_1) \subseteq \add(x_2)$ implies that $y_1 \in \add(y_2)$, which is a contradiction.  \\

Secondly, suppose that $H$ has a realization to a PDFG ring. In this case, all $M_i$ and $N_i$ are finitely generated.
Therefore, for each $1 \leq i \leq 2$, there exist finite numbers $\alpha_i$ and $\beta_i$ such that $x_i = \alpha_i y_1 + \beta_i y_2$.
If both $\alpha_1$ and $\alpha_2$ are zero, then $x_i = \beta_i y_2$. 
Given that $\{ x_1, x_2 \}$ is a generating set, this implies that $\{y_2\}$ is a generating set for 
 $H$. However, this contradicts the fact that $H$ is non-cyclic.
 Therefore, at least one of $\alpha_1$ and $\alpha_2$, and similarly one of $\beta_1$ and $\beta_2$, is nonzero.
From the equation $n_1 x_2 = n_2 x_1 + n_3 x_2$, we conclude that $n_3$ is nonzero; otherwise, $\add(x_1) = \add(x_2)$, leading to a contradiction.
Recall that $n_2 $  is non-zero and we have the equation:

\[ n_1 (\alpha_2 y_1 + \beta_2 y_2) = n_2 (\alpha_1 y_1 + \beta_1 y_2) + n_3 (\alpha_2 y_1 + \beta_2 y_2) \]

If $\beta_2 = 0$, then $\beta_1$ is nonzero, and from the above equation, we conclude that $y_2 \in \add(y_1)$, which is a contradiction. 
Consequently, $\beta_2 \neq 0$. Similarly, $\alpha_2 \neq 0$.

Thus, both $y_1$ and $y_2$ are in $\add(x_2)$. Since $H$ is not cyclic and cannot be generated solely by $x_1$, 
we conclude that $x_2$ appears in a form of either $y_1$ or $y_2$ regarding to 
generating  set $\{x_1, x_2\}$. Therefore, $\add(x_2) \subseteq \add(y_i)$ for either $i = 1$ or $i = 2$. In any case, this is a contradiction, 
as both $y_1$ and $y_2$ belong to $\add(x_2)$ and therefore lie in $\add(y_i)$.
\end{proofenumerate} 
\end{proof}

\section{Realization to hereditary VNR rings }

The aim of this section is understanding when $H$, which is a two generated
$\aleph_0$-monoid, 
  has realization to a VNR hereditary ring. 
{Let $M$ be a reduced commutative monoid.
We write $x \leq y$ whenever $y = x + z$ for some $z \in H$.
 An element $p \in M$ is called {\defit{prime}} if
 whenever $p \leq x+ y$, then $p \leq x$ or $p \leq y$. If every nonzero element of $M$ can be 
 written as a sum of prime elements, then $M$ is called 
 {\defit{primely generated}}.
  By \cite[Theorem 6.8]{22}, every finitely generated refinement monoid is primely generated.
  Recall that $M$  is {\defit{separative}} if the following cancellation of elements holds for every $a, b, c \in M$:
\begin{equation}  
a + c = b + c  {\text { and }} c \in \add(a) \cap \add(b)  \implies  a = b  \tag{Eq 4.1}  \label{Eq 4.1}
\end{equation}

Moreover, $c \in M$ is called \defit{cancelable} if for every $a, b \in M$:
\[
a + c = b + c \implies a = b
\]

When $H = \langle {x, y} \rangle_{\aleph_0}$ is an $\aleph_0$-monoid, by $H_{\text{f}}:= \langle x, y \rangle$, we refer to the submonoid generated by $x$ and $y$. In the next lemma, we examine properties of $H_{\text{f}}$ when it is a refinement monoid.

\begin{lemma} \label{clear}
Let $H = \langle {x, y} \rangle_{\aleph_0}$   such that   $H_{\text{f}} $ is refinement. Then:
\begin{enumerate}
\item Suppose that   $ n,   m, k, k' $ are finite integers such that 
$n > m$ and $n x + k y  = m x + k' y $, then $(n - m + 1) x + k y  =  x + k' y$. 
\item If $H$ is of {\rm Type (1)} regarding to $\{x, y\}$ and $ n_1 x + m_1 y   = n_2 x  +  m_2 y$, where all $n_i , m_i$  are nonzero finite integers. Then 
$n_1 x = n_2 x$ and $ m_1 y = m_2 y $. 
\item Suppose that 
\[ ny = mx + n'y, \]
where \( n \neq 0 \), \( m \neq 0 \), and \( n' \) are finite integers. Then \( x \) is cancelable. Moreover, if \( H \) is of {\rm Type (2)} regarding  to \(\{x, y\}\), then \( x \leq y \).

\end{enumerate}
\end{lemma}
\begin{proof}
\begin{enumerate}
\item 
By [22, Theorem 4.5], every primely generated refinement monoid is separative. In
particular every finitely generated refinement monoid is separative. This shows easily that (1) 
holds due to \ref{Eq 4.1}. 

\item Suppose that $H$ is of Type (1) and 
$ n_1 x + m_1 y   = n_2 x  +  m_2 y$, where all $n_i , m_i$  are nonzero finite integers. 
Since $H_{\text{f}}$ is refinement, this 
implies that $n_1 x = x' + y'$, $m_1 y = z' + t'$, $n_2 x = x' + z'$, and $m_2 y = y' + t'$, for $x', y', z', t' \in H_{\text{f}}$. 
This situation can be represented in the form of a square:

\begin{center}
\medskip
\begin{tabular}{|c||c|c|c|}\hline
& $n_2 x$  & $m_2 y$ \\\hline\hline
$n_1 x$ & $x'$ & $y'$ \\\hline
$m_1 y$ & $z'$ & $t'$ \\\hline
\end{tabular}
\medskip
\end{center}

Each $x', y', z', t'$ has a form regarding to $\{x, y\}$, and since $H$ is of Type (1), there exist finite integers $k, k'$ such that $x' = kx$ and $t' = k'y$. We present the following table:

\begin{center}
\medskip
\begin{tabular}{|c||c|c|c|}\hline
& $n_2 x$  & $m_2 y$ \\\hline\hline
$n_1 x$ & $kx$ & $0$ \\\hline
$m_1 y$ & $0$ & $k' y$ \\\hline
\end{tabular}
\medskip
\end{center}

This implies that $n_1 x = kx = n_2 x $ and $ m_1 y = k' y = m_2 y$. 

\item Suppose that $ ny = mx + n'y,$ where  $0 \neq m,  0 \neq n, n'$ are finite integer.  Thus 
$x \in \operatorname{add}(y)$. 
As we mentioned in Part (1), $H_{\text{f}}$ is separative, so when $x \in \operatorname{add}(y)$, it has to be cancelable. Now suppose that $H$ is of Type (2) regarding to $\{x, y\}$. 
  Since  $\add(x) \neq \add(y)$, $n'$ must be nonzero. 
Therefore, we can suppose that $n$ is the smallest possible nonzero integer such that $ny = mx + n'y$, for nonzero finite integers $m, n, n'$. We claim $n = 1$. If not,
 we write $(n - 1)y + y = mx + n'y$. Then, since $H_{\text{f}}$ is a refinement monoid, there exist $x', y', z', t'$ such that

\begin{center}
\medskip
\begin{tabular}{|c||c|c|c|}\hline
& $m x $  & $n' y $ \\\hline\hline
$(n - 1) y$ & $x'$ & $y'$ \\\hline
$ y$ & $z'$ & $t'$ \\\hline
\end{tabular}
\medskip
\end{center} 

That is $mx = x' + z'$, $n'y = y' + t'$, $(n - 1) y = x' + y'$, and $y = z' + t'$.
We claim that $x' = 0$.  
Since $y \notin \operatorname{add}(x)$ and $mx = x' + z'$, if $x'$ is nonzero, then in  a  form of $x'$ regarding to $\{x, y\}$ only $x$ can appear. 
But then, due to the minimality of $n$, and since $(n - 1) y = x' + y'$ and $\operatorname{add}(x) \neq \operatorname{add}(y)$, in a form of $y'$,  $y$ should appear. 
Then due to the minimality of $n$, we get to a contradiction.
This proves the claim and so  $x' = 0$. 
 But then $z'$ has to be nonzero and is equal to $kx$, for some finite $k \geq 1$. 
Now, again, the fact that $\operatorname{add}(x) \neq \operatorname{add}(y)$ implies that $t'$ is nonzero. However, this is in contradiction with the minimality of $n$. 
That means $n = 1$.  
\end{enumerate}
\end{proof}}

A strong source of examples of hereditary  VNR rings with interesting properties can be found in rings bases on  weighted graphs and separated graphs.
 In the next results, we will understand when a two-generated $\aleph_0$-monoid can be realized to 
  algebras based on these graphs.
 Let us recall the definitions.

\begin{defi}\label{weighted} {\rm \cite[Definition 5.1]{Roozbeh} }
A {\it weighted graph} $E=(E^0,E^{\st}, E^1,r,s,w)$ consists of three countable sets, $E^0$ called {\it vertices}, $E^{\st}$ {\it structured edges} and $E^1$ {\it edges} and
 maps $s,r:E^{\st}\rightarrow E^0$,  and  a {\it weight map} $w:E^{\st} \rightarrow \mathbb N$ such that $E^1= \coprod_{\alpha \in E^{\st}} \{\alpha_i \mid 1\leq i \leq w(\alpha)\}$, i.e., 
 for any $\alpha\in E^{\st}$, with $w(\alpha)=k$, there are $k$ distinct elements $\{a_1,\dots,\alpha_k\}$, and $E^1$ is the disjoint union of  all such sets for all $\alpha \in E^{\st}$. 
\end{defi}

We refer to \cite[Definition 5.2]{Roozbeh} for weighted Leavitt path $K$-algebra, where $K$ is a field, 
 constructed by the above graph. This algebra is denoted by $\LL_K(E,w)$.
Before recalling the main theorem from \cite{Roozbeh}, let us recall that it is always supposed that the graph is row finite, that means for every vertex $v, $
$s^{-1} (v) $ is a finite set. Moreover, $k_v$ denotes the number of different weights in $s^{-1}(v)$.

\begin{teor}  {\rm (\cite[Remark 12]{Pif} and \cite [Theorem 5.12]{Roozbeh})} \label{kspcts} Let $K$ be a field and $E$ be a  weighted graph and $k_{v} \leq 1$, for every vertex $v$. 
 Suppose that $M_E$ is  the abelian monoid generated by $\{v \mid v \in E^0 \}$ subject to the relations 
\begin{equation}\label{phgqcu}
n_vv=\sum_{\{\alpha\in E^{\st} \mid s(\alpha)=v \}} r(\alpha),
\end{equation}
for every $v\in E^0$ that emits edges, where $n_v= \max\{w(\alpha)\mid \alpha\in E^{\st}, \, s( \alpha)=v\}$. Then there is a natural monoid isomorphism $\V(\LL_K(E,w))\cong M_E$. Furthermore, if $E$ is finite, then $\LL_K(E,w)$ is hereditary. 
\end{teor}


When $E$ is a graph, we represent the number of edges from vertex $u$ to $v$ as $d(u, v)$.
Let $E$ be a weighted graph with two vertices, $u$ and $v$: 

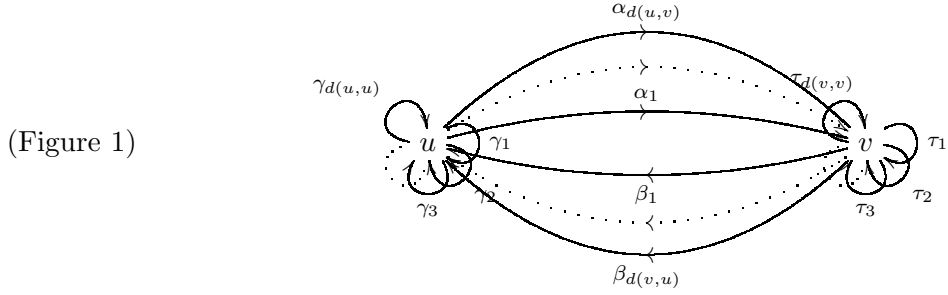
\[
\xymatrix  @=150pt  { 
  u \ar@/^3.5pc/@{}[r]^{{\alpha}_{d(u, v)}} |{\SelectTips{cm}{}\object@{>}}  \ar@/^2.4pc /@{.>}[r] |{\SelectTips{cm}{}\object@{>}} 
  \ar@/^1.0 pc/[r]^{\alpha_{1}} |{\SelectTips{cm}{}\object@{>}} \ar@{.}@(l,d) \ar@(ur,dr)^{\gamma_{1}} \ar@(r,d)^{\gamma_{2}} \ar@(dr,dl)^{\gamma_{3}} 
\ar@(l,u)^{\gamma_{d(u, u)}}
        &
v   \ar@/^1.0 pc/[l]^{\beta_1}|{\SelectTips{cm}{}\object@{>}}   \ar@/^2.5pc /@{.>}[l] |{\SelectTips{cm}{}\object@{>}}  \ar@/^3.5pc/[l]^{\beta_{d(v, u)}} |{\SelectTips{cm}{}\object@{>}}
\ar@{.}@(l,d) \ar@(ur,dr)^{\tau_{1}} \ar@(r,d)^{\tau_{2}} \ar@(dr,dl)^{\tau_{3}} 
 \ar@(l,u) ^(0.6) {\tau_{d(v, v)}}  }   
 \label{graph} \tag{Figure 1}
\]

\begin{remark} \label{MHowisME}
$M_E$ is the monoid generated by $u, v$ and with the following relations:

\begin{enumerate} [itemsep=5pt]
  \item $n_{u} u = d{(u, u)} u + d{(u, v)} v$

  \item $n_{v} u = d{(v, v)} u + d{(v, u)} v$
\end{enumerate}
\end{remark}

Where,  $n_u = \max \{ w(a) \vert a \in \{ \alpha _i  \vert  1 \leq i \leq d(u, v)  \} 
\cup \{ \gamma_i \vert  1 \leq i \leq d(u, u)  \} $ and 
 $n_v = \max \{ w(a) \vert a \in \{ \beta _i  \vert  1 \leq i \leq d(v, u)  \} 
\cup \{ \tau_i \vert  1 \leq i \leq d(v, v)  \} $ \\

Recall that $(E,w)$  is called \defit{acyclic} if there is no cycle in $(E,w)$. 
An edge in $(E,w)$  is \defit{unweighted} if its weight is $1$ and  \defit{weighted} otherwise.
$E^1_w$ denotes the set of all weighted edges in $(E,w)$.
When $(E,w)$ is weighted, we mean always that 
 $E^1_w$ is nonempty.

{

\begin{teor}
The following statements for $H$ are equivalent:
\begin{enumerate} [itemsep=5pt]
    \item $H$ has a realization to a regular weighted Leavitt path algebra with two vertices.
    \item $H$ is hereditary, and there exists a generating set $\{x_1, x_2\}$ and 
$1 \leq i \neq j \leq 2$ such that  $x_i = nx_j$, for some finite 
$n \geq 2$. Moreover, all other relations are derived from $x_i = nx_j$. 
 
\end{enumerate}
In this case, for every generating set $\{y_1, y_2\}$, there is only one infinite form and $H$ is of {\rm Type (3)}.  

\end{teor}}
\begin{proof}
$(1) \implies (2)$.  Let $H \cong_{\aleph_0} \V ^{ \aleph_0}(\LL_K(E,w))$, where $\LL_K(E,w)$ is regular weighted Leavitt path algebra. 
Then, using \cite [Theorem  5.1]{me}, there is an isomorphism of monoids, $f: \V(\LL_K(E,w)) \to \add(x)$, where $x$ is a
 braiding element in $H$. 
Now we claim that $E$ has two vertices $u, v$ such that 
 $d(u, u) = 0 = d(v, v)$ and one of the following holds:
\begin{enumerate} [itemsep=5pt]
    \item [(a)] $d(u, v) = 0$ and  $d(v, u) = 1$.
    \item [(b)] $d(v, u) = 0$ and  $d(u, v)  = 1$.
\end{enumerate}

Before proving the claim, recall that if $u, v \in E^0$ and there is a path $p$ such that $s(p) = u$ and $r(p) = v$, then we write $u \geq v$. 
If $u\in E^0$, then $T(u):=\{v\in E^0 \mid u\geq v\}$ is called the \defit{tree} of $u$. If $X\subseteq E^0$, we define $T(X):=\bigcup_{v\in X}T(v)$. 
Recall that $(E,w)$ is called \defit{well-behaved} if Conditions (LPA1), (LPA2), (LPA3), (W1), and (W2) from \cite{RAIMUNDPREUSSER},  as explained in the following, hold.
Following \cite[Theorem 6.3.2]{RAIMUNDPREUSSER}, $\LL_K(E,w)$ is VNR if and only if $(E,w)$ is acyclic and well-behaved.

Proof of claim: \\
Since  $\LL_K(E,w)$ is VNR,  it is acyclic, so $d(u, u) = 0 = d(v, v)$, and either $d(v, u) = 0$ or $d(u, v) = 0$. 

As we mentioned, $(E,w)$ satisfies (LPA1) and  (W2),  that is:

\begin{enumerate}
    \item [(LPA1)]: {\it Any vertex $x \in E_0$ emits at most one weighted edge.}
   
    \item [(W2)]: {\it There is no $n\geq 1$ and paths $p_1,\dots,p_n,q_1,\dots,q_n$ in $(E,w)$ such 
    that $r(p_i)=r(q_i)~(1\leq i\leq n)$, $s(p_1)=s(q_n)$, $s(p_i)=s(q_{i-1})~(2\leq i \leq n)$ and for any $1\leq i \leq n$, the
     first letter of $p_i$ is a weighted edge, the first letter of $q_i$ is an unweighted edge, and the last letters of $p_i$ and $q_i$ are distinct.}
    \smallskip
    \end{enumerate}

  If $d(u, v) = 0$, then there is no ${\alpha}_i$, and following (LPA1), there is at most one $1 \leq i \leq d(v, u)$ such that $\beta_i$ is weighted. That means   $E^1_w$ has exactly  one element. 
  Then since there is no loop, the only path that starts with a weighted edge is $\beta_i$.  
If   $\beta_l$ is another edge,  (which is unweighted), then  we have two 
  paths $\beta_i$ and $\beta_l$ contradicting (W2).
 That means (a) should hold.
 Similarly, if $d(v, u) = 0$, then (b) holds.
So we have proved the claim.  
Then, by Remark \ref{MHowisME}, $M_E$ is the monoid generated only by the relation, either 
$ n_{u}u = v $ or  $ n_{v}v = u $.  This proves (2).  \\

$(2) \implies (1)$.
Assuming that assumption (2) holds and 
there exists a generating set $\{x_1, x_2\}$ and 
$1 \leq i \neq j \leq 2$ such that  $x_i = nx_j$, for some finite 
$n \geq 2$.
 Utilizing Remark \ref{MHowisME}, we construct the graph
 $(E, w)$  such that $H_{\text{f}}$ and $M_E$ are isomorphic as a monoid. 
 Specifically, we suppose that $d(u, u) = d(v, v) = d(v, u) = 0$. 
In Figure 1, the graph has only one arrow $\alpha_1$, which is weighted with weight $n$. Note that by  \cite[Theorem 5.1]{me}, $\text{add}(x_1 + x_2)$ is equal to $H_f$, and so by \cite[Corollary 4.6]{me}, $\V(\mathcal{L}_K(E,w))$ is $\aleph_0$-isomorphic to $H$. Now we need to show that $\mathcal{L}_K(E,w)$ is regular. Since $d(u, u) = d(v, v) = d(v, u) = 0$, then $(E,w)$ is acyclic, and clearly (LPA1) and (W2) hold. It remains to check conditions (LPA2), (LPA3), and (W1) as explained below:

\begin{enumerate}
    \item [(LPA2)]: Any vertex $ x \in T(r(E^1_w))$ emits at most one edge.
    \smallskip
    
    Since $E^1_w$  contains only one  edge, then 
     $r(E^1_w)$  contains only one element $v$. The same holds for $T(r(E^1_w))$.  Since $d(v, v) = 0 = d(v, u)$, (LPA2) simply holds. 

    \item [(LPA3)]: If two weighted edges $e,f\in E^1_w$ are not in line, then $T(r(e))\cap T(r(f))=\emptyset$.
    \smallskip
    
   Recall that  two edges $e$ and $f$ are {\defit {in line}} if they are equal or there is a path from $r(e)$ to $s(f)$ or 
   from $r(f)$ to $s(e)$. Thus, clearly, (LPA3) holds; indeed, there is only one  weighted edge.
    
    \item [(W1)]: No cycle in $(E,w)$ is based at a vertex $v\in T(r(E^1_w))$.
    \smallskip
    
    This automatically holds since $(E,w)$ is acyclic and there is no cycle in $(E,w)$.
\end{enumerate}

For the last part of the theorem, observe that whenever $x_1 = nx_2$ or $nx_1 = x_2$, then $\add (x_1) = \add (x_2)$. 
 Thus, $H$ is of Type (3). Therefore, using Theorem \ref{typeseriuse} and \cite[Theorem 5.5 (Part 2)]{me}, there exists only one infinite form with respect to any generating set.

\end{proof}

{

In the special case of a graph with weights of 1 (or unweighted), this construction yields the usual Leavitt path algebras (LPAs).
Recall that a vertex $v$ in a graph is called \defit{regular} if  $1 \leq \vert s^{-1} (v) \vert < \infty $ and 
$M_E$ of a graph is the commutaive monoid generated by $\{ a_v \vert v \in E_0\}$ and with relations 
$$ a_v = \sum _{e \in s^{-1} (v)} r (e)$$
for each regular vertex $v \in E$. 
One may wonder if a two-generated $\aleph_0$-monoid 
 $H$ can have a realization to a (weighted) LPA  of a graph with only one vertex and countably many loops:

\begin{remark}
 Let $E$ be the infinite rose graph:

\begin{equation}
\xymatrix{
& \bullet\ar@{<.}@(l,d)  \ar@(ur,dr)^{e_1} \ar@(r,d)^{e_2} \ar@(dr,dl)^{e_3} &
}
\label{E} \tag{E}
 \end{equation}

According to \cite[Theorem 3.1.10]{bookaragene},
 LPA of $E$, denoted by $L_K(E)$, is a purely infinite simple ring, and some properties of $\V(L_K(E))$ are understood in
 \cite[Corollary 2.3.]{Ktheory}. We use \cite[Example 11 and Theorem 14]{GENE} to see that $\V(R)$ is exactly $M_F$, where $F$ is the row finite equivalence of $E$:

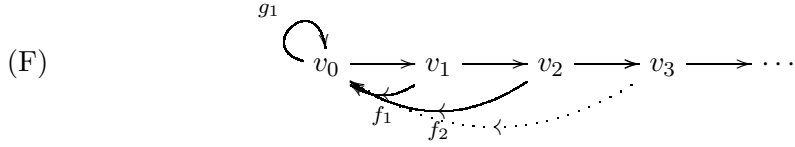
\begin{equation}
\xymatrix{
    v_0  \ar@(l,u) ^{g_1}\ar[r]  &  v_1 \ar[r]  \ar@/^1.0 pc/[l]^{f_1}|{\SelectTips{cm}{}\object@{>}}    &
 v_2 \ar[r] \ar[r]  \ar@/^1.5 pc/[ll]^{f_2}|{\SelectTips{cm}{}\object@{>}}    &v_3 \ar[r]  
\ar@/^2.0pc /@{.>}[lll] |{\SelectTips{cm}{}\object@{>}} 
 &    \cdots  & 
 }
\label{F} \tag{F}
 \end{equation}

Consequently, $M_F$ is generated by $\{a_{v_i} \mid i \geq 0\}$ such that $a_{v_i} = a_{v_0} + a_{v_{i + 1}}$ for every $i \geq 0$. In particular, it is not finitely generated.
For a better understanding of $M_F$, note that $a_{v_1} = a_{ v_0} + a_{v_2} = 
(a_{v_0} + a_{v_1}) + a_{ v_2} = a_{v_0} + a_{v_2} + a_{v_1} = a_{v_1} + a_{v_1}$.

If $\Bbb{Z}$ is the ring of integers considered to be a monoid with addition and $z$ is a distinct element from the elements of $\Bbb{Z}$, then $\Bbb{Z} \sqcup {z}$ denotes the monoid with an operation compatible with the operation of $\Bbb{Z}$, where $z$ has the property that $z + x = x + z = x$ for every $x \in \Bbb{Z}$ or $x = z$.

We can see that $M_F$ is the monoid isomorphic to $\Bbb{Z} \sqcup {z}$. To see this, define $f: M_F \to  \Bbb{Z}  \sqcup  {z}$ such that $f(a_{v_i}) = i - 1$. Then $f(n a_{v_0}) = -n$, and by this, it is easy to see that $f$ is a monoid isomorphism. 
\end{remark}
}

To conclude this section and provide a characterization of realizable two-generated monoids to VNR rings, we recall separated graphs:




 
\begin{defi} \label{defsepgraph} {\rm  \cite [Definition 2.1] {ARAGOOD}}
A \defit{separated graph} is a pair $(E,C)$ where $E$ is a graph,  $C=\bigsqcup
_{v\in E^ 0} C_v$, and
$C_v$ is a partition of $s^{-1}(v)$ (into pairwise disjoint nonempty
subsets) for every vertex $v$. (In case $v$ is a sink, we take $C_v$
to be the empty family of subsets of $s^{-1}(v)$.)

If all the sets in $C$ are finite, we shall say that $(E,C)$ is a \emph{finitely separated} graph.

The constructions we introduce revert to existing ones in case $C_v= \{s^{-1}(v)\}$
for each non-sink $v\in E^0$. We refer to a \emph{non-separated graph} or a \emph{trivially separated graph} in that situation.
\end{defi}
 
{

Given a finitely separated graph $(E,C),$
 we define the monoid of the separated graph $(E,C)$
 to be the commutative monoid given by generators and relations as follows:

$$M(E, C) := \langle a_v := \sum_{e \in X} a_{r(e)} \text{ for every } X \in C_v, v \in E^0 \rangle$$

 For the definition of Leavitt path algebra of separated graphs denoted by $L_K (E, C)$, $K$ is a field,  we refer to \cite [Definition 2.2.] {ARAGOOD}. 
 We recall from \cite{ARAGOOD} that  $\V (L_K (E, C)) \cong M(E,C)$. 
 When the graph is adaptable (see Definition \ref{adap}), then $L_K(E, C)$ is hereditary VNR, as demonstrated in \cite{AraBosaPardo}.
 To recall 
 adaptable  ones, note that for a directed graph $E =(E_0,E_1,s,r)$, we have:
\begin{itemize}
\item [(i)] Similar to the case of weighted graphs, we define a pre-order on $E_0$  (the path-way pre-order) by $ v \leq  w$ 
 if and only if there is a directed path $\gamma$ in $E$  with $s(\gamma) = w$ and $r(\gamma) = v$.

\item [(ii)]  Let $ \sim $ be the equivalence relation on the set $E_0$ defined, for 
every $v, w \in  E_0$, by  $v \sim  w$  if  $v \leq  w$  and $ w \leq v.$
 Set $I = E_0/ \sim$ , so that the preorder $\leq $ on $E_0$
 induces a partial order on $I$.
 We will also denote by $\leq$
 this partial order on $I $. Thus, denoting by $[v]$ 
the class of $v\in E_0$,  we have 
$[v] \leq [w]$ if and only if $v \leq  w.$
 We will often refer to $[v]$  as the strongly connected component of $v.$

\item [(iii)] We say that $E$ is strongly connected if every two vertices of $ E_0$
 are connected through a directed path, i.e., if I is a singleton.
\end{itemize}

\begin{defi} { \rm \cite[Definition 1.2] {AraBosaPardo}} \label{adap}
Let $(E, C)$ be
a finitely separated graph and let $(I,\leq )$ be the partially ordered set associated to the pre-ordered set 
$( E_0 , \leq ).$ We say that 
$( E , C )$ is \defit{adaptable} if $I$ is finite, and there exist a 
partition $I = I_{\text {free}} \cup  I_{\text {reg}}$, 
and a family of subgraphs $\{E_p\}_ {p\in I}$ 
 of $E$
 such that the following conditions are satisfied: 
 \begin{enumerate}
 \item \label{here} $E_0 = \bigsqcup _{p \in I} E^0 _p ,$ where $E_p$ is a strongly connected row-finite graph 
 if $p \in  I_{\text{reg} }$ 
and $E^0 _p = \{v^p \}$  is a single vertex if $p \in  I_{\text{free} }$.

\item \label{here1} For $p \in I_{\text {reg}}$ and $w \in  E^0 _p$ , we have
 that  $\vert C_w \vert   = 1$  and $\vert s_{ E_p} ^{- 1} (w) \vert  \geq 2.$ Moreover,
all edges departing from $w$ either belong to the graph $E_p$ or connect $w$ to a vertex
$u \in  E^0_q,$ with $q < p$  in $I .$

\item \label{here2} For $p \in  I_{\text{free} } $, we have that
$ s^{-1} (v^p) = \emptyset$ if and only if $p$ is minimal in $I$.
 If $p$ is not
minimal, then there is a positive integer $k( p)$
 such that  $ C_{v^p} = \{X_1 ^{(p)} , . . . , X_{k( p)} ^{(p)} \}$.
Moreover, each $X_i ^{(p)}$ 
is of the form
$$ X_i ^{(p)} = \{ \alpha ( p, i),  \beta ( p, i, 1), \beta( p, i, 2), \cdots , \beta( p, i, g( p, i))\},$$
for some  $g( p, i) \geq  1,$ where $ \alpha ( p, i)$  is a loop, i.e., 
$s(\alpha ( p, i)) = r (\alpha( p, i)) = v^p ,$ 
 and $r (\beta ( p, i, t)) \in  E ^0_q$  for $q < p$ in $I$.
 Finally, we have  $E^ 1_p = \{\alpha( p, 1), . . . , \alpha( p, k( p))\}$.
 \end{enumerate}

\end{defi}

We understand finitely separated adaptable graphs with two vertices in the following example:

\begin{example}
Consider  the following graph: 

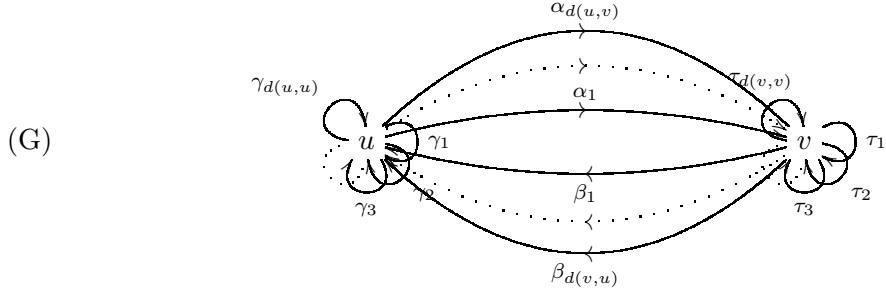
\[
\xymatrix  @=150pt  { 
  u \ar@/^3.5pc/@{}[r]^{\alpha_{d(u, v)}} |{\SelectTips{cm}{}\object@{>}}  \ar@/^2.4pc /@{.>}[r] |{\SelectTips{cm}{}\object@{>}} \ar@/^1.0 pc/[r]^{\alpha_{1}} |{\SelectTips{cm}{}\object@{>}} \ar@{.}@(l,d) \ar@(ur,dr)^{\gamma_{1}} \ar@(r,d)^{\gamma_{2}} \ar@(dr,dl)^{\gamma_{3}} 
\ar@(l,u)^{\gamma_{d(u, u)}}
        &
v   \ar@/^1.0 pc/[l]^{\beta_1}|{\SelectTips{cm}{}\object@{>}}   \ar@/^2.5pc /@{.>}[l] |{\SelectTips{cm}{}\object@{>}}  \ar@/^3.5pc/[l]^{\beta_{d(v, u)}} |{\SelectTips{cm}{}\object@{>}}
\ar@{.}@(l,d) \ar@(ur,dr)^{\tau_{1}} \ar@(r,d)^{\tau_{2}} \ar@(dr,dl)^{\tau_{3}} 
 \ar@(l,u) ^(0.6) {\tau_{d(v, v)}}  }   \\
 \label{graph} \tag{G}
\]

To make $G$ adaptable, we have the following possibilities:
\medskip
\begin{enumerate}
    \item $\vert I \vert = 1$. In this case, both
     $d{(u, v)}$ and $d{(v, u)}$ are nonzero. Also, by Definition \ref{adap}\ref{here}, $I_{\text{free}} = \emptyset$ and $I_{\text{reg}} = \{ [u] = [v] \}$. Then $E_{[u]}^{0} = \{ u, v \}$ and, by
      Definition \ref{adap}\ref{here1}, $\vert C_u \vert = \vert C_v \vert = 1$, $d{(u, u)} + d{(u, v)} \geq 2$ and $d{(v, u)} + d{(v, v)} \geq 2$. Thus $M_E$ is the monoid graph of a non-separated one generated by $a_u, a_v$ and relations
    \[
    \begin{aligned}
        a_{u} &= d{(u, v)} a_v  + d{(u, u)} a_u, \\
        a_{v} &= d{(v, u)} a_u  + d{(v, v)} a_v.
    \end{aligned}
    \]

    \item $\vert I \vert = 2$. In this case, we may assume that $d{(v, u)} = 0$. Then the following cases happen:
    \begin{enumerate}[label=(\roman*)]
\medskip
        \item $d{(u, v)} = 0$. Thus, by  Definition \ref{adap}\ref{here2}, both $u$ and $v$ are regular and consequently, by  Definition \ref{adap}\ref{here1},
         we have $u = d{(u, u)} u$ and $v = d{(v, v)} v$, where
          $d{(u, u)}, d{(v, v)} \geq 2$.
\medskip
        \item $d(u, v) \neq 0$. Then $u > v$ and so $v$ is minimum. Both $E^{0}_u$ and $E^{0}_v$ have only one element. Then the four cases happen:

                 \begin{enumerate}[label=(\roman{enumi} $a_{\arabic*}$)]
\medskip
            \item $u$ and $v$ are both regular. Then by  Definition \ref{adap}\ref{here1},
             our graph is non-separated and $d{(u, u)}, d{(v, v)} \geq 2$. Thus, the relations are:
            \[
            \begin{aligned}
                a_{u} &= d{(u, v)} a_v  + d{(u, u)} a_u, \\
                a_{v} &= d{(v, v)} a_v.
            \end{aligned}
            \]
\medskip
            \item $u$ is regular and $v$ is free. Since $s^{-1}(u)$ is nonempty, and $v$ is minimal, by  Definition \ref{adap}\ref{here2}, $u \in E^{0}_{u}$. 
            Then by  Definition \ref{adap}\ref{here1}, $d{(u, u)} \geq 2$ and 
            $d{(v, v)} = 0$. Consequently, the relation is:
            \[
            \begin{aligned}
                a_{u} &= d{(u, v)} a_v  + d{(u, u)} a_u, \text{ with } d{(u, u)} \geq 2.
            \end{aligned}
            \]
\medskip
            \item Let $u$ be free and $v$ be regular. If $v \in E^{0}_u$, then  Definition \ref{adap}\ref{here2}
             implies that 
            there should be an arrow from $v$ to $u$ which is a contradiction. Then $v \in E^{0}_v$, and so by 
             Definition \ref{adap}\ref{here1}, 
             $d{(v, v)} \geq 2$ and $a_v = d{(v, v)} a_v$. 
            Since $u$ is not minimal, by  Definition \ref{adap}\ref{here2}, there exists a positive integer $k(u)$ such that $d(u, u) = k(u)$. 
             Also, we have relations $a_u = a_u + i a_v$, for every $1 \leq i \leq k(u)$.
\medskip
            \item Let $u, v$ be free. Since $s^{-1} (u) \neq \emptyset$ and $v$ is minimum in $I$, we conclude by  Definition \ref{adap}\ref{here2}
             that $u \in E^{0}_u$. Then 
            from $E^{0}_v = \{ v\}$ and  Definition \ref{adap}\ref{here2}, 
            we conclude that  
            $d{(v, v)} = 0$ and there exists $k(u) \geq 1$ and relations
            \[
                a_u =  a_u + i a_v, \text{ for every } 1 \leq i  \leq k(u).
            \]
        \end{enumerate}
    \end{enumerate}
\end{enumerate}

\end{example}


The example helps to understand two-generated $\aleph_0$-monoid  with realization  to VNR hereditary rings: } {

\begin{teor} \label{last}
Let $H$ be a non-cyclic  $\aleph_0$-monoid generated by $\{ x_1, x_2\}$. Then the following statements are equivalent:
\begin{enumerate}
\item $H$ has a realization to a hereditary VNR  ring.
\item $ H_{\text{f}}:= \langle x_1, x_2 \rangle  = \add (x_1 + x_2)$  is a refinement monoid and
$H$ is braiding over $H_{\text{f}}$. 
\item 
\begin{enumerate}[label=\textup(\roman*\textup),leftmargin=2.5em]
\item \label{hcc:threeinf} \label{hcc:first} $n x_i + \aleph_0 x_j = \aleph_0 x_i + \aleph_0 x_j$ with $n$ finite implies 
$\aleph_0 x_j = \aleph_0 x_i +\aleph_0 x_j$ and $x_i \in \add (x_j)$
  and if $H$ is of {\rm Type (2)}, then $x_i \leq x_j$.
\item \label{hcc:twoinf} Whenever    $ n \geq  m $ and $n x_i + k x_j = m x_i + k' x_j $, then $(n - m + 1) x_i + k x_j =  x_i + k' x_j$. 
\item If $x_i \not \in \add(x_j)$,  $ n \geq  m $ and $m x_i + \aleph_0 x_j = n x_i + \aleph_0 x_j$, then there exist $k , k'$ such that $n x_i + k x_j = m x_i + k' x_j $.
   In this case: 
\begin{itemize}
\item [(a)] If  $H$ is of {\rm Type (2)}, then 
$x_j$ is cancelable  and consequently when 
 $ k > k'$  then $(n - m + 1) x_i + (k - k' ) x_j =  x_i $ and if  $k' >  k$, then $(n - m + 1) x_i  =  x_i + (k' - k ) x_j$. 
\item [(b)] If $H$ is of Type (1), then $ n x_i  =  m x_i$ and $k x_j =  k' x_j $. 
\end{itemize}    
\item \label{hcc:finite-infinite} \label{hcc:last} An element of $H$ cannot have both finite and infinite forms.
\end{enumerate}
\end{enumerate} 
\end{teor}

\begin{proof}
\begin{proofenumerate}
 \item [$(1) \implies (2).$] 
 Suppose that $H$ has a realization to a hereditary VNR ring $R$,
then  $ H_{\text{f}}:= \langle x_1, x_2 \rangle  \cong \V(R)$
and $H$ is braiding over $H_{\text{f}}$, by \cite [Theorem 5.1]{me}.  
Thus $ H_{\text{f}} $ is a refinement monoid.
 \item [$(2) \implies (3).$]   By Corollary \ref{useful}, $H$ is hereditary. 
 Part (iv) follows 
 from 
Theorem \ref{hereditarycasecor}. 
Now if $n x_i + \aleph_0 x_j = \aleph_0 x_i + \aleph_0 x_j$ with $n$ finite, then by 
Theorem \ref{hereditarycasecor}, $x_i \in \add(x_j)$ and if $H$ is of Type (2), using   Part (3) of 
Lemma \ref{clear}, we see that $x_i \leq x_j$.  This shows (i) holds.
Part  (ii) is exactly Part (1) of  Lemma \ref{clear}.  
Now let 
 $x_i \not \in \add(x_j)$,  $ n \geq  m $ and $m x_i + \aleph_0 x_j = n x_i + \aleph_0 x_j$, 
then  by Theorem \ref{hereditarycasecor}, there exist $k , k'$ such that $n x_i + k x_j = m x_i + k' x_j $.
Then if $H$ is of Type (2), then (a) follows from the fact that $x_j$ is  cancelable based on 
 Part (3) of 
Lemma \ref{clear}. Part (b) is Part (2) of Lemma \ref{clear}.

\item [$(3) \implies (1).$] Suppose that (3) holds. Then  by 
Theorem \ref{hereditarycasecor}, $H$ is hereditary and using \cite [Theorem 5.1]{me} to conclude (1),
it is enough to  show that $\operatorname{add}(x_1 + x_2)$ is  refinement. 

We consider two cases:

\begin{enumerate}
\item[T1:] If $H$ is of Type (1) and $n_1 x_1 + m_1 x_2 = n_2 x_1 + m_2 x_2$, then from $n_1 x_1 + \aleph_0 x_2 = n_2 x_1 + \aleph_0 x_2$ and Part (iii)(b) and also Part (ii), 
we have that $x_1 = (\vert n_1 - n_2 \vert + 1) x_1$ and 
 $x_2 = (\vert m_1 - m_2 \vert + 1) x_2$. Let $k_1, k_2$ be the smallest positive integers such that $x_1 = k_1 x_1$ and $x_2 = k_2 x_2$. 
We claim that $\operatorname{add}(x_1 + x_2)$ is a monoid graph of a separated graph with two vertices $u, v$ such that $1 = \vert C_v \vert = \vert C_u \vert$.

To see this, note that $x_i = k_i^m x_i$, for every $1 \leq i \leq 2 $ and $m \geq 1$. If $x_i = n x_i$,  where $n \geq  k_i$, 
then we write $n = m_j k_i^j + \cdots + m_1 k_i + m_0$, for some positive integers $m_j, \cdots, m_0$ and $j \geq 1$. Then $n x_i = (m_j + \cdots + m_0) x_i$.
 We do the same for $n' := m_j + \cdots + m_0$, if $n' \geq k_i$. By repeating this process if it is necessary and bearing in mind that $k_i$ is the minimum one 
 such that $x_i = k_i x_i$, we conclude that $\operatorname{add}(x_1 + x_2)$ is generated by $\{ x_1, x_2\}$ and relations $x_i = k_i x_i$.

\item[T2:] If $H$ is of Type (2), or (3), then for every $a, b \in \langle x_1, x_2\rangle$, we have either $a \leq b $ or  $b \leq a $. Also by condition of (3), we see that $H_{\text{f}}$ 
is separative and consequently by \cite[Theorem 3.5]{22},  it is refinement. 
\end{enumerate}
\end{proofenumerate}
\end{proof}

}

\bibliographystyle{hyperalphaabbr}
\bibliography{kappa_monoids}

\end{document}